\documentclass[11pt,reqno]{amsart}
\usepackage{cmap}
\usepackage[T1]{fontenc}
\usepackage[utf8]{inputenc}
\usepackage[english]{babel}
\usepackage{amsfonts,amssymb,amsthm,amsmath}
\usepackage{url}
\usepackage{enumitem}

\usepackage[scr=rsfso]{mathalfa}

\usepackage{secdot}

\usepackage{mathtools}

\usepackage{graphicx}
\usepackage{epstopdf}
\usepackage{a4wide}
\usepackage{tikz}
\usepackage{xcolor}
\usepackage[colorlinks=true,linktocpage,pdfpagelabels,
bookmarksnumbered,bookmarksopen]{hyperref}
\definecolor{ForestGreen}{rgb}{0.1,0.6,0.05}
\definecolor{EgyptBlue}{rgb}{0.063,0.1,0.6}
\hypersetup{
	colorlinks=true,
	linkcolor=EgyptBlue,         
	citecolor=ForestGreen,
	urlcolor=olive
}

\usepackage[hyperpageref]{backref}

\newtheorem{theorem}{Theorem}
\newtheorem{proposition}[theorem]{Proposition}
\newtheorem{lemma}[theorem]{Lemma}

\theoremstyle{definition}
\newtheorem{definition}[theorem]{Definition}
\newtheorem{remark}[theorem]{Remark}

\usepackage{mathtools}
\mathtoolsset{showonlyrefs}
\allowdisplaybreaks

\let\OLDthebibliography\thebibliography
\renewcommand\thebibliography[1]{
	\OLDthebibliography{#1}
	\setlength{\parskip}{1pt}
	\setlength{\itemsep}{1pt plus 0.3ex}
}

\numberwithin{equation}{section}
\numberwithin{theorem}{section}
\numberwithin{equation}{section}
\numberwithin{theorem}{section}

\newcommand{\X}{X_0^s(\Omega)}

\title[Nonuniqueness for fractional parabolic equations]{Nonuniqueness for fractional parabolic equations\\ with sublinear power-type nonlinearity}

\author[J.~Benedikt]{Ji\v{r}\'{i} Benedikt}
\address{
	Department of Mathematics and NTIS, Faculty of Applied Sciences,
	\newline\indent 
	University of West Bohemia, Univerzitn\'i 8, 301 00 Plze\v{n}, Czech Republic
}
\email{benedikt@kma.zcu.cz}

\author[V.~Bobkov]{Vladimir Bobkov}
\address{
	Institute of Mathematics, Ufa Federal Research Centre, RAS,
	\newline\indent 
	Chernyshevsky str. 112, 450008 Ufa, Russia
}
\email{bobkov@matem.anrb.ru}

\author[R.~N.~Dhara]{Raj Narayan Dhara}
\address{Department of Mathematical Analysis and
	Applications of Mathematics
	\newline\indent 
	Faculty of Science, Palack\'y University
	\newline\indent 
	17.listopadu 12,
	771 46 Olomouc, Czech Republic
}
\email{dhara@math.muni.cz}

\author[P.~Girg]{Petr Girg}
\address{
	Department of Mathematics and NTIS, Faculty of Applied Sciences,
	\newline\indent 
	University of West Bohemia, Univerzitn\'i 8, 301 00 Plze\v{n}, Czech Republic
}
\email{pgirg@kma.zcu.cz}

\date{}

\subjclass[2010]{
	35A01,  
	35A02,  
	35B30,  
	35K58,  
	35R11.  
}
\keywords{Fractional Laplacian, initial-boundary value problem, non-Lipschitz reaction term, nonuniqueness.}

\thanks{
	J.~Benedikt and P.~Girg were supported by the Grant Agency of the Czech Republic (GA\v CR) under Grant No.~22-18261S.
	R.~N.~Dhara was supported IBS PAN, Warszawa, Poland and faculty project IGA\_PrF\_2022\_008.
}

\begin{document}
	\begin{abstract} 
		We show that the parabolic equation $u_t +
		(-\Delta)^s u = q(x) |u|^{\alpha-1} u$ posed in a time-space cylinder $(0,T) \times \mathbb{R}^N$ and coupled with zero initial condition and zero nonlocal Dirichlet condition 
		in $(0,T) \times (\mathbb{R}^N \setminus \Omega)$, where $\Omega$ is a bounded domain, 
		has at least one nontrivial nonnegative finite energy solution provided $\alpha \in (0,1)$ and the nonnegative bounded weight function $q$ is separated from zero on an open subset of $\Omega$. 
		This fact contrasts with the (super)linear case $\alpha \geq 1$ in which the only bounded finite energy solution is identically zero.
	\end{abstract} 
	\maketitle 
	
	\section{Introduction}
	
	Let $\Omega$ be a bounded Lipschitz domain in $\mathbb{R}^N$, $N \geq 1$. 
	Let $s \in (0,1)$, $\alpha > 0$, and $T>0$.
	Our aim is to investigate the (non)uniqueness of finite energy solutions of the model problem
	\begin{equation}\label{eq:P}
		\left\{
		\begin{aligned}
			u_t
			+
			(-\Delta)^s u &= q(x) |u|^{\alpha-1} u
			&&\text{in}~ (0,T) \times \Omega,\\
			u &= 0 &&\text{in}~  (0,T) \times (\mathbb{R}^N \setminus  \Omega),\\
			u &= 0 &&\text{at}~ t=0,
		\end{aligned}
		\right.
	\end{equation}
	where $(-\Delta)^s$ is the fractional Laplacian corresponding to the pointwise definition
	\begin{equation}\label{eq:deltas}
		(-\Delta)^s u(x) = 
		-
		\frac{s 2^{2s} \Gamma(\frac{N+2s}{2})}{\pi^\frac{N}{2} \Gamma(1-s)}
		\lim_{\varepsilon \to 0+} \int_{\mathbb{R}^N \setminus B_\varepsilon(x)} \frac{u(y)-u(x)}{|y-x|^{N+2s}} \, \textrm{d}y.
	\end{equation}
	We impose the following assumptions on the weight function $q$:
	\begin{enumerate}[label={($\mathcal{Q}_{\arabic*}$)}]
		\addtolength{\itemindent}{0em}
		\item\label{Q1} $q \in L^\infty(\Omega)$ and $q \geq 0$ in $\Omega$. 
		\item\label{Q2} There exists an open subset $O$ of $\Omega$ and a constant $q_0>0$ such that $q \geq q_0$ a.e.\  in $O$.
	\end{enumerate} 
	
	In the ``local'' case $s=1$, in which $(-\Delta)^s$ 	
	formally corresponds to the usual Laplace operator, the systematic investigation of the problem \eqref{eq:P} was initiated in the seminal work of \textsc{Fujita} \& \textsc{Watanabe} \cite{FW}, where the authors established the existence of at least one nontrivial nonnegative solution, in addition to the trivial solution of the problem, in the sublinear case $\alpha \in (0,1)$.
	Since \eqref{eq:P} with $s=1$ is autonomous, it can be shown that any positive time shift of the solution  constructed in \cite{FW} is also a solution and, thus, there is a continuum of nontrivial solutions.		
	In this way, the nonuniqueness for \eqref{eq:P} with $s=1$ is reminiscent of that for the simple Cauchy problem
	\begin{equation}\label{eq:Cauch}
		\left\{
		\begin{aligned}
			\frac{\textrm{d}y}{\textrm{d}t}(t)
			&=  |y(t)|^{\alpha-1} y(t),
			&&t > 0,\\
			y(0) &= 0.
		\end{aligned}
		\right.
	\end{equation}
	In fact, \cite{FW} covers more general boundary conditions than in \eqref{eq:P} and a more general right-hand side of the form $q(x) f(u)$, where the function $f: [0,M] \to [0,+\infty)$ is nondecreasing, concave, and satisfies $f(0)=0$, $f>0$ on $(0,M]$, and $\int_{0}^M \textrm{d}u/f(u) <+\infty$ for some $M>0$, the latter assumption meaning the violation of the Osgood condition. 
	The concavity of $f$ has been recently dropped in \cite{LRS2017}.
	We also refer to \cite{BS,mlak,RW} for nonuniqueness results for similar initial-boundary value problems with some particular choices of the nonlinearity. 
	On the other hand, if $f$ does satisfy the Osgood condition $\int_{0}^M \textrm{d}u/f(u) = +\infty$ (and $f$ is not necessarily concave), \cite{FW} establishes the uniqueness of the trivial solution in the class of bounded solutions.
	
	The results of \cite{FW} were later generalized in \cite{BBGKT,BGKT} to the case of the $p$-Laplacian, where it appears that the (non)uniqueness is a matter of interplay between the reaction and nonlinear diffusion terms, cf.\ \cite{PV} for related results for reaction-diffusion equations of porous medium type.
	Since the fractional Laplacian \eqref{eq:deltas} also represents a nonclassical diffusion (see, e.g., the overview \cite{vaz1}), it is interesting to know whether results of \cite{FW,LRS2017} can be extended to the nonlocal problem \eqref{eq:P}.
	In \cite[Section~3]{LRS2017}, the authors \textit{discuss} a possibility of such  extension. 
	However, various aspects of the problem driven by nonlocal settings seem to be not systematically well developed in the literature.
	In particular, this applies to the relation between several natural notions of solutions of \eqref{eq:P}, such as pointwise, weak, finite energy, semigroup solutions, etc. 
	The method of proof of our main result, Theorem~\ref{thm:1} formulated below, is substantially different from that used in \cite{FW,LRS2017}, namely, it does not rely on the semigroup theory and provides a nonuniqueness result in the class of so-called finite energy solutions (see Section~\ref{sec:preliminaries}), which might be a priori different from the class of semigroup solutions.
	We also refer to \cite{peral1,diaz,GT,LPPS} for the existence and qualitative properties of solutions to some related nonlocal parabolic problems.
	
	Let us formulate our main result. 
	All necessary definitions will be introduced in the subsequent section. 
	\begin{theorem}\label{thm:1}
		Let $\alpha \in (0,1)$ and let the assumptions \ref{Q1}, \ref{Q2} be satisfied.
		Then there exists $T>0$ such that \eqref{eq:P} has a nontrivial nonnegative bounded finite energy solution $u$. 
		
		Moreover, for any time shift $\tau>0$, the function $\tilde{u}$ defined as $\tilde{u}(t,x)=0$ for $t \in (0,\tau]$ and $\tilde{u}(t,x) = u(t-\tau, x)$ for $t \in (\tau,T+\tau)$ is a nontrivial nonnegative bounded finite energy solution of \eqref{eq:P} on $(0,T+\tau)$.
	\end{theorem}
	
	Although Theorem~\ref{thm:1} produces a continuum of nontrivial finite energy solutions of \eqref{eq:P}, the (non)uniqueness of nontrivial solutions \textit{modulo time shift} is unknown to us even in the ``local'' unweighted case $s=1$ and $q \equiv 1$ a.e.\ in $\Omega$. 
	This poses an interesting open problem, taking into account, e.g., the results for the $p$-Laplacian established in \cite{BGKT}.
	
	The nonuniqueness result of 
	Theorem \ref{thm:1} is complemented by the following facts.
	\newpage
	\begin{proposition}\label{prop:uniq}
		Let $\alpha \geq 1$ and let the assumptions \ref{Q1}, \ref{Q2} be satisfied.
		The following assertions hold:
		\begin{enumerate}[label={\rm(\roman*)}]
			\item\label{prop:uniq:1}
			If $\alpha=1$, then the trivial solution is the unique solution in the class of finite energy solutions of \eqref{eq:P}.
			\item\label{prop:uniq:2}
			If $\alpha > 1$, then the trivial solution is the unique solution in the class of bounded finite energy solutions of \eqref{eq:P}.
		\end{enumerate}
	\end{proposition}
	
	The rest of the paper has the following structure.
	In Section~\ref{sec:preliminaries}, we provide necessary preliminaries such as a function spaces framework and definitions of finite energy solutions, sub- and supersolutions.	
	Section~\ref{sec:proof} is devoted to the proof of Theorem~\ref{thm:1}. 
	More precisely, in Sections~\ref{sec:subsol} and~\ref{sec:supresol}, we construct a subsolution and a supersolution of \eqref{eq:P}, respectively. The existence of a solution of \eqref{eq:P} in between the pair of these sub- and supersolutions is shown in Section~\ref{sec:existence}.
	Finally, in Section~\ref{sec:uniq}, we prove Proposition~\ref{prop:uniq}.

	\section{Preliminaries}\label{sec:preliminaries}
	Recalling that $s\in (0,1)$, we consider the  fractional Sobolev space 
	\begin{equation}\label{eq:defH}
		H^s(\mathbb{R}^N) 
		= 
		\left\{
		u \in L^2(\mathbb{R}^N):~ [u]_{H^{s}(\mathbb{R}^N)}<+\infty
		\right\}
	\end{equation}
	endowed with the norm
	$$
	\|u\|_{H^s(\mathbb{R}^N)} := \|u\|_{L^2(\mathbb{R}^N)} + [u]_{H^{s}(\mathbb{R}^N)},
	$$
	where 
	\begin{align*}
		\|u\|_{L^2(\mathbb{R}^N)} := \left(\int_{\mathbb{R}^N} |u|^2 \, \textrm{d}x\right)^{1/2}
		\quad \text{and} \quad
		[u]_{H^{s}(\mathbb{R}^N)}
		:=
		\left(\iint_{\mathbb{R}^{N}\times\mathbb{R}^{N}}\frac{|u(x)-u(y)|^{2}}{|x-y|^{N+2s}}\, \textrm{d}x\textrm{d}y
		\right)^{1/2}
	\end{align*}
	are the standard Lebesgue norm and the Gagliardo seminorm, respectively.
	In order to deal with the homogeneous nonlocal Dirichlet condition in a weak formulation of \eqref{eq:P}, we consider a subspace of $H^s(\mathbb{R}^N)$ consisting of functions supported in $\Omega$.
	Following \cite{BRS}, we denote this subspace as $X_0^s(\Omega)$, i.e.,
	$$
	X_0^s(\Omega) 
	:= 
	\left\{
	u \in H^s(\mathbb{R}^N):~ u=0 ~\text{a.e.\ in}~ \mathbb{R}^N \setminus \Omega
	\right\}.
	$$
	It is known (see, e.g, \cite[Lemma~1.29]{BRS} for the case $N > 2s$ and \cite[Lemma~8.1]{CGH2020} for the case $N=1$ and $s \in [1/2,1)$) that $X_0^s(\Omega)$ is a Hilbert space with the scalar product
	\begin{align*}
		\mathcal{E}(u,v)
		:= 
		\frac{c_{N,s}}{2}
		\iint_{\mathbb{R}^{N}\times\mathbb{R}^{N}}\dfrac{(u(x)-u(y))(v(x)-v(y))}{|x-y|^{N+2s}}\, \textrm{d}x\textrm{d}y
	\end{align*}
	and the associated norm
	\begin{equation}\label{eq:Euu}
		\|u\|_{X_0^s(\Omega)}
		:=
		\mathcal{E}(u,u)^{1/2}
		=
		\sqrt{\frac{c_{N,s}}{2}}\, [u]_{H^{s}(\mathbb{R}^N)},
	\end{equation}
	where the constant $c_{N,s}$ is the same as in \eqref{eq:deltas}, i.e.,
	\begin{equation}\label{eq:cns}
		c_{N,s}
		:=
		\frac{s 2^{2s} \Gamma(\frac{N+2s}{2})}{\pi^\frac{N}{2} \Gamma(1-s)} 
		> 0.
	\end{equation}
	
	\begin{remark}\label{rem:reg}
		The embedding $X_0^s(\Omega) \hookrightarrow L^2(\Omega)$ is compact. 
		In the case $N>2s$, we refer to \cite[Lemma~1.30]{BRS} for an explicit statement. 
		The remaining case $N \leq 2s$, which reduces to $N=1$ and $s \in [1/2,1)$, can be covered as follows.
		Observe that if $0<\sigma<\sigma'<1$, then
		$H^{\sigma'}(\mathbb{R}^N) \hookrightarrow H^\sigma(\mathbb{R}^N)$ continuously, see, e.g., \cite[Proposition~2.1]{hitch}.
		Clearly, this result gives the continuous embedding $X_0^{\sigma'}(\Omega) \hookrightarrow X_0^{\sigma}(\Omega)$.
		Considering now $s \in [1/2,1)$ and choosing any $\sigma \in (0,1/2)$, we obtain the chain of embeddings
		$$
		X_0^{s}(\Omega) 
		\stackrel{\text{continuously}}{\hookrightarrow}
		X_0^{\sigma}(\Omega)
		\stackrel{\text{compactly}}{\hookrightarrow}
		L^2(\Omega),
		$$
		which gives the compactness of the embedding $X_0^{s}(\Omega) \hookrightarrow L^2(\Omega)$ in the case $N \leq 2s$. 
		We also note that the embedding $X_0^s(\Omega) \hookrightarrow L^2(\Omega)$ is dense since 
		$C_{0}^{\infty}(\Omega)$ is a subspace of $X_0^{s}(\Omega)$ and, at the same time, $C_{0}^{\infty}(\Omega)$
		is dense in $L^2(\Omega)$. 
	\end{remark}
	
	It follows from \cite[Proposition~3.1~f)]{BRS} 
	that
	the sequence of all (properly normalized) eigenfunctions of the fractional Laplacian forms 
	an orthonormal basis
	of the space $X_0^s(\Omega)$ and thus this space
	is separable. 	
	Although the statement of \cite[Proposition~3.1~f)]{BRS} requires $N > 2s$, it can be straightforwardly checked that the result remains valid also in the case $N=1$ and $s \in [1/2,1)$, thanks to Remark~\ref{rem:reg}.
	Since $X_0^s(\Omega)$ is a Hilbert space, it is 
	reflexive and thus $(X_0^s(\Omega))^*$, the dual space to $X_0^s(\Omega)$, is also separable by, e.g.,
	\cite[Proposition~1.3]{roub}.	
	We refer the reader to \cite{AVW} for some other properties of $(X_0^s(\Omega))^*$.
	
	Identifying $L^2(\Omega)$ with its dual and recalling that the embedding $X_0^s(\Omega) \hookrightarrow L^2(\Omega)$ is compact and dense, and $X_0^s(\Omega)$ is a separable Hilbert space, we define \textit{Gelfand’s triple} (or, equivalently, an \textit{evolution triple})  $X_0^s(\Omega) \hookrightarrow L^2(\Omega) \hookrightarrow (X_0^s(\Omega))^*$, see, e.g., \cite[Section~23.4]{zeid} or \cite[Section~7.2]{roub}.
	
	Hereinafter, we denote by $\langle \cdot,\cdot \rangle$ the duality pairing between $(X_0^s(\Omega))^*$ and $X_0^s(\Omega)$.
	Identifying a function $u \in L^2(\Omega)$ with a corresponding element from $(X_0^s(\Omega))^*$, we have
	\begin{equation}\label{eq:prod2}
		\langle u,v \rangle = \int_\Omega u v \,\textrm{d}x
		\quad \text{for any}\quad v \in X_0^s(\Omega),
	\end{equation}
	see, e.g., \cite[Eq.~(7.14)]{roub}.
	In general, for a Banach space $B$, we will write $\langle \cdot,\cdot \rangle_{B^* \times B}$ for the duality pairing between the dual space $B^*$ and $B$.
	
	In order to introduce the definition of finite energy solution of the problem \eqref{eq:P}, we define the Bochner spaces
	$$
	L^2((0,T) \to X_0^s(\Omega))
	\quad \text{and} \quad
	L^2((0,T) \to (X_0^s(\Omega))^*)
	$$
	in the standard way, following, e.g., \cite[Section~1.5]{roub} by setting $I = (0,T)$, $p=2$, $V = X_0^s(\Omega)$, and using $V^* = (X_0^s(\Omega))^*$. 
	Then $L^2((0,T) \to X_0^s(\Omega)) = L^2(I;V)$ and $L^2((0,T) \to (X_0^s(\Omega))^*) = L^2(I;V^*)$ in the notation of \cite{roub}.
	Notice that $L^2((0,T) \to X_0^s(\Omega))$ is a Hilbert space with the scalar product defined as $\int_0^T \mathcal{E}(v(t,\cdot), w(t,\cdot))\,\textrm{d}t$, see, e.g., \cite[Proposition~23.2~(e)]{zeid}.
	More generally, since $(X_0^s(\Omega))^*$ and $X_0^s(\Omega)$ are separable, we know from \cite[Proposition~1.38]{roub} that $L^2((0,T) \to (X_0^s(\Omega))^*)$ is the dual space to $L^2((0,T) \to X_0^s(\Omega))$, and their duality pairing is given by the formula
	$$
	\langle f,u \rangle_{L^2((0,T) \to (X_0^s(\Omega))^*) \times L^2((0,T) \to X_0^s(\Omega))}
	=
	\int_0^T \langle f(t), u(t,\cdot) \rangle \, \textrm{d}t.
	$$
	Taking any $u \in L^2((0,T) \to X_0^s(\Omega))$, we define its 
	distributional derivative $u_t$ in the standard way as in \cite[Eq.~(7.2)]{roub}, cf.\ \cite[Section~23.5]{zeid} and \cite[Appendix, pp.~1044-1052]{zeid2}.
	If we further assume that $u_t \in L^2((0,T) \to (X_0^s(\Omega))^*)$ (or, equivalently, $u \in W^{1,p,p'}(I;V,V^*)$ in the notation of \cite[Eq.~(7.1)]{roub} with $p,I,V,V^*$ as above and $p'=\frac{p}{p-1}=2$), then  $u \in C([0,T] \to L^2(\Omega))$, see \cite[Lemma~7.3]{roub}.

	Consider the auxiliary boundary value problem
	\begin{equation}\label{eq:Pf}
		\left\{
		\begin{aligned}
			u_t
			+
			(-\Delta)^s u &= f
			&&\text{in}~ (0,T) \times \Omega,\\
			u &= 0 &&\text{in}~  (0,T) \times (\mathbb{R}^N \setminus  \Omega),\\
			u &= 0 &&\text{at}~ t=0,
		\end{aligned}
		\right.
	\end{equation}
	where $f \in L^2((0,T) \to (X_0^s(\Omega))^*)$. 
	For practical purposes, it would be also convenient to work with $f \in L^2((0,T)\times \Omega)$. 
	In this case, we identify such $f$ with a corresponding element from $L^2((0,T) \to (X_0^s(\Omega))^*)$, cf.\ \cite[Example~23.4]{zeid}.
	We introduce the notion of \textit{finite energy} solution of \eqref{eq:Pf} following the  approach presented in \cite{LPPS}, see also \cite[Sections~7.1, 7.2]{roub}.
	\begin{definition}
		A function $u$ is called a finite energy solution of \eqref{eq:Pf} if $u \in L^2((0,T) \to X_0^s(\Omega))$, $u_t \in L^2((0,T) \to (X_0^s(\Omega))^*)$, $\|u(t,\cdot)\|_{L^2(\Omega)} \to 0$ as $t \to 0+$,    
		and the equality
		\begin{equation}\label{eq:weak-f}
			\int_0^T \langle u_t(t),\varphi(t,\cdot) \rangle \, \textrm{d}t
			+
			\int_0^T \mathcal{E}(u(t,\cdot), \varphi(t,\cdot)) \, \textrm{d}t
			=
			\int_0^T \langle f(t), \varphi(t,\cdot) \rangle \, \textrm{d}t
		\end{equation}
		is satisfied for all $\varphi \in L^2((0,T) \to X_0^s(\Omega))$.
	\end{definition}
	
	We observe that by the Riesz representation theorem (see, e.g., \cite[Proposition~21.17~(a)]{zeid}), there exists a linear bijective operator $J: \X \to (X_0^s(\Omega))^*$ such that $\|Ju\|_{(X_0^s(\Omega))^*} = \|u\|_{\X}$ and
	\begin{equation}\label{eq:e=inner}
		\mathcal{E}(u,\varphi) = \langle Ju, \varphi \rangle
		\quad
		\text{for any}~ u, \varphi \in \X.
	\end{equation}
	If, in addition to $u \in X_0^{s}(\Omega)$, we assume $(-\Delta)^s u \in L^2(\mathbb{R}^N)$, then $Ju = (-\Delta)^s u$ and the relation \eqref{eq:prod2} applies.
	This claim follows from  \cite[Theorem~5.3]{kwasnic} (see the equivalence of realizations $L_I$ and $L_Q$ of the fractional Laplacian).
	
	Let now $u$ be a finite energy solution of \eqref{eq:Pf}.
	Then $u(t,\cdot) \in \X$ for a.e.\ $t \in (0,T)$.
	In view of \eqref{eq:e=inner}, we have $J{u}(t) \in (X_0^s(\Omega))^*$, $\|Ju(t)\|_{(X_0^s(\Omega))^*} = \|u(t,\cdot)\|_{\X}$, and 
	\begin{equation}\label{eq:e=scal}
		\mathcal{E}(u(t,\cdot), \varphi(t,\cdot))
		=
		\langle J{u}(t), \varphi(t,\cdot)
		\rangle
	\end{equation}
	for a.e.\ $t \in (0,T)$ and all $\varphi \in L^2((0,T) \to X_0^s(\Omega))$.
	Substituting \eqref{eq:e=scal} into \eqref{eq:weak-f}, we get
	\begin{equation}\label{eq:weak-f-eq1}
		\int_0^T 
		\langle 
		u_t(t)
		+
		J{u}(t)
		-
		f(t),
		\varphi(t,\cdot)
		\rangle \,
		\textrm{d}t
		=
		0
		\quad \text{for all}~
		\varphi \in L^2((0,T) \to X_0^s(\Omega)).
	\end{equation}
	In particular,  \eqref{eq:weak-f-eq1} holds for test functions of the form $\varphi(t,x) = \psi(t) \xi(x)$, where $\psi \in C_0^\infty(0,T)$ and $\xi \in \X$.
	Therefore, arguing exactly as in \cite[Proof (I) of Proposition~23.10]{zeid}, we deduce that any finite energy solution $u$ of \eqref{eq:Pf} satisfies
	\begin{equation}\label{eq:weak-f-eq2}
		\langle 
		u_t(t)
		+
		Ju(t)
		-
		f(t),
		\xi
		\rangle 
		=
		0
		\quad
		\text{for a.e.}~ t \in (0,T)
		~\text{and all}~
		\xi \in \X.
	\end{equation}
	The identity \eqref{eq:weak-f-eq2} will be used later in the proof of Theorem~\ref{thm:1} in Section~\ref{sec:existence} to establish a key estimate for the time derivative of approximate solutions of \eqref{eq:P}.
	
	It is known from, e.g, \cite[Theorem~26]{LPPS} that \eqref{eq:Pf} possesses a unique finite energy solution. 
	Moreover, again by \cite[Theorem~26]{LPPS}, the weak maximum principle holds, i.e., if $f$ is nonnegative in the distributional sense, then the finite energy solution is nonnegative.  
	Since the problem \eqref{eq:Pf} is linear, the maximum principle is equivalent to the comparison principle. 
	The weak comparison principle will be important in our consideration of the nonlinear problem \eqref{eq:P}.

	\medskip
	Now we introduce the notion of finite energy solution, subsolution, and supersolution of \eqref{eq:P}. 
	\begin{definition}\label{def:sol-parab}
		A function $u$ is called a finite energy solution of \eqref{eq:P} if $u \in L^2((0,T) \to X_0^s(\Omega))$, $u_t \in L^2((0,T) \to (X_0^s(\Omega))^*)$, $\|u(t,\cdot)\|_{L^2(\Omega)} \to 0$ as $t \to 0+$,    
		and the equality
		\begin{equation}\label{eq:weak}
			\int_0^T \langle u_t(t),\varphi(t,\cdot) \rangle \, \textrm{d}t
			+
			\int_0^T \mathcal{E}(u(t,\cdot), \varphi(t,\cdot)) \, \textrm{d}t
			=
			\int_0^T \int_\Omega q(x) |u(t,x)|^{\alpha-1} u(t,x) \varphi(t,x) \,\textrm{d}x \textrm{d}t
		\end{equation}
		is satisfied for all $\varphi \in L^2((0,T) \to X_0^s(\Omega))$.
	\end{definition}
	\begin{definition}\label{def:subsuper}
		A function $u$ is called a finite energy subsolution of \eqref{eq:P} if $u \in L^2((0,T) \to X_0^s(\Omega))$, $u_t \in L^2((0,T) \to (X_0^s(\Omega))^*)$, $\|u(t,\cdot)\|_{L^2(\Omega)} \to 0$ as $t \to 0+$,    
		and \eqref{eq:weak} 
		is satisfied with ``$\leq$'' in place of ``$=$''  for all nonnegative $\varphi \in L^2((0,T) \to X_0^s(\Omega))$.
	\end{definition}
	\begin{definition}\label{def:super}
		A function $u$ is called a finite energy supersolution of \eqref{eq:P} if $u \in L^2((0,T) \to X_0^s(\Omega))$, $u_t \in L^2((0,T) \to (X_0^s(\Omega))^*)$, $\|u(t,\cdot) - u_0(\cdot)\|_{L^2(\Omega)} \to 0$ as $t \to 0+$, where $u_0 \in L^2(\Omega)$ and $u_0 \geq 0$ a.e.\ in $\Omega$, 
		and \eqref{eq:weak} 
		is satisfied with ``$\geq$'' in place of ``$=$'' 
		for all nonnegative $\varphi \in L^2((0,T) \to X_0^s(\Omega))$.
	\end{definition}
	
	In the full generality, sub- and supersolutions do not have to be zero in $(0,T)\times (\mathbb{R}^N \setminus \Omega)$, and subsolutions do not need to satisfy $\|u(t,\cdot)\|_{L^2(\Omega)} \to 0$ as $t \to 0+$.
	However, our more restrictive Definitions~\ref{def:subsuper} and~\ref{def:super} are sufficient for the purposes of the present paper. 
	
	Let us also note that for $\alpha \in (0,1]$ 
	the right-hand side of \eqref{eq:weak} is well-defined for any $u, \varphi \in L^2((0,T) \to X_0^s(\Omega))$. 
	Indeed, we have $u, \varphi \in L^2((0,T)\times \Omega)$, and in view of \ref{Q1} and the boundedness of $\Omega$, we get
	\begin{equation}\label{eq:weak:est1}
		\left|\int_0^T \int_\Omega q(x) |u|^{\alpha-1} u \varphi \,\textrm{d}x \textrm{d}t\right| 
		\leq
		C \int_0^T \int_\Omega (1+|u|) |\varphi| \,\textrm{d}x \textrm{d}t < +\infty.
	\end{equation}

	\section{Proof of Theorem~\ref{thm:1}}\label{sec:proof}
	\subsection{Construction of a subsolution}\label{sec:subsol}
	As usual, we denote by $B_R(x_0)$ an open ball in $\mathbb{R}^N$ of radius $R>0$ centered at $x_0 \in \mathbb{R}^N$.
	For given $p>-1$, $R>0$, and $x_0 \in \mathbb{R}^N$, we define a function $\psi_{x_0,R}: \mathbb{R}^N \to [0,+\infty)$ as
	\begin{equation}\label{eq:subsol-1}
		\psi_{x_0,R}(x) = (R^2-|x-x_0|^2)_+^p,
	\end{equation}
	where the subindex ``+'' denotes the positive part of the function. 
	In particular, we have $\text{supp}\,\psi_{x_0,R} = B_R(x_0)$.
	It is known from \cite[Theorem 1, Eq.~(1.5)]{dyda1} (see also \cite[Corollary~3 (i)]{DKK}) that
	\begin{equation}\label{eq:dsphi}
		(-\Delta)^s\psi_{0,1}(x)
		=
		\kappa\, 
		{}_2F_1
		\left(
		s+\frac{N}{2}, -p+s; \frac{N}{2}; |x|^2
		\right),
		\quad |x|<1,
	\end{equation}
	where $\kappa:=-
	\frac{c_{N,s} B(-s,p+1) \pi^{N/2}}{\Gamma(N/2)} > 0$,  
	$c_{N,s}$ is given by \eqref{eq:cns},
	${}_2F_1$ is the Gauss hypergeometric function whose third argument, $N/2$, is positive and hence the right-hand side of \eqref{eq:dsphi} is well defined (see, e.g., \cite[Section~15.1]{abram}), and $B,\Gamma$ are the  beta and gamma functions, respectively (see, e.g., \cite[Section~6]{abram}).
	We also consider a function $\theta: [0,+\infty) \to [0,+\infty)$ which is a solution of the Cauchy problem
	\begin{equation}\label{eq:subsol-2}
		\left\{
		\begin{aligned}
			\frac{\mathrm{d} \theta}{\mathrm{d} t}(t) &= \frac{q_0R^{\alpha-1}}{2} \theta^\alpha(t) &&\text{for}~ t >0,\\
			\theta(t) &>0 &&\text{for}~ t >0,\\ 
			\theta(0) &= 0.
		\end{aligned}
		\right.
	\end{equation}
	Evidently, $\theta$ is increasing and can be expressed in the form
	\begin{equation}\label{eq:subsol0}
		\theta(t) = \frac{1}{R}\left(\frac{(1-\alpha) q_0}{2} \,t\right)^\frac{1}{1-\alpha}, \quad t \geq 0.
	\end{equation}
	Let us now define a function $\underline{w}: [0,T] \times \mathbb{R}^N \to [0,+\infty)$ as
	\begin{equation}\label{eq:subsol1}
		\underline{w}(t,x)
		=
		\theta(t) \psi_{x_0,R}(x).
	\end{equation}
	Our aim is to show that $\underline{w}$ is a finite energy subsolution of \eqref{eq:P} for a sufficiently small $T>0$.
	
	We start with the following auxiliary information regarding $(-\Delta)^s \psi_{x_0,R}$, cf.\ \cite[Lemma~9]{dyda1}.
	\begin{lemma}\label{lem:Psi}
		Let $p>2s$. 
		Then the following assertions are satisfied:
		\begin{enumerate}[label={\rm(\roman*)}]
			\item\label{lem:Psi:2} $(-\Delta)^s \psi_{x_0,R} \in C(B_R(x_0)) \cap L^\infty(B_R(x_0))$.
			\item\label{lem:Psi:1} $(-\Delta)^s \psi_{x_0,R} \in L^2(\mathbb{R}^N)$.
			\item\label{lem:Psi:5} 
			For all $v \in H^s(\mathbb{R}^N)$, we have
			\begin{equation}\label{eq:kwasnic00}
				\mathcal{E}(\psi_{x_0,R}, v)
				=
				\int_\Omega
				(-\Delta)^s\psi_{x_0,R}(x) v(x) \, \textrm{d}x.
			\end{equation}
			\item\label{lem:Psi:3} $(-\Delta)^s \psi_{x_0,R} \leq 0$ in $\mathbb{R}^N \setminus B_R(x_0)$.
			\item\label{lem:Psi:4} There exists a constant $C \geq \kappa R^{-2s} > 0$ such that
			\begin{equation}\label{eq:frac1}
				\frac{(-\Delta)^s\psi_{x_0,R}(x)}{\psi_{x_0,R}(x)} \leq C \quad \text{for all}~ x \in B_R(x_0).
			\end{equation}
		\end{enumerate}
	\end{lemma}
	\begin{proof}
		Thanks to the translation and dilation properties of the fractional Laplacian (cf.\ \cite[Lemma 2.6]{garofalo}), it is sufficient to assume that $R=1$ and $x_0$ is the origin of $\mathbb{R}^N$. 
		We will denote $B_r(0)$ as $B_r$, for brevity.
		
		\ref{lem:Psi:2}
		We use the formula \eqref{eq:dsphi}.
		It is known that the hypergeometric series defining the function ${}_2F_1
		\left(
		s+\frac{N}{2}, -p+s; \frac{N}{2}; \tau
		\right)$ converges absolutely for any $\tau \in [0,1]$ since $p>2s$, see, e.g., \cite[(15.1.1)]{abram}.
		Then, Abel's theorem guarantees that   ${}_2F_1
		\left(
		s+\frac{N}{2}, -p+s; \frac{N}{2}; \cdot
		\right) \in C([0,1])$. 
		This implies that $(-\Delta)^s \psi_{0,1} \in C(B_1) \cap L^\infty(B_1)$.
		
		\ref{lem:Psi:1}
		We start by showing that $(-\Delta)^s \psi_{0,1} \in L^2(\mathbb{R}^N\setminus B_1)$.
		For any $x \in \mathbb{R}^N\setminus B_1$, we use \eqref{eq:deltas} to get
		\begin{align}
			\notag    
			(-\Delta)^s \psi_{0,1}(x)
			&=
			-c_{N,s}
			\lim_{\varepsilon \to 0+} \int_{\mathbb{R}^N \setminus B_\varepsilon(x)} \frac{\psi_{0,1}(y)}{|y-x|^{N+2s}} \, \textrm{d}y
			\\
			&=
			-c_{N,s}
			\lim_{\varepsilon \to 0+} \int_{B_1 \setminus B_\varepsilon(x)} \frac{(1-|y|^2)^p}{|y-x|^{N+2s}} \, \textrm{d}y
			\leq
			0,
			\label{eq:delta-psi}
		\end{align}
		where $c_{N,s}$ is given by \eqref{eq:cns}.
		First, we discuss the case $|x| = R \geq 2$. Thanks to the estimates $|x-y| \geq |x|-|y| \geq R - 1 \geq 1$ and $(1-|y|^2)^p \leq 1$ for $y \in B_1$, we deduce from \eqref{eq:delta-psi} that
		$$
		|(-\Delta)^s \psi_{0,1}(x)| \leq 
		c_{N,s}
		\lim_{\varepsilon \to 0+} \int_{B_1 \setminus B_\varepsilon(x)} \frac{1}{(R-1)^{N+2s}} \, \textrm{d}y
		=
		\frac{c_{N,s} |B_1|}{(R-1)^{N+2s}}.
		$$
		Passing to the spherical coordinates and noting that $R-1 \geq R/2$ for any $R \geq 2$, we arrive at
		\begin{align*}
			\int_{\mathbb{R}^N \setminus B_2}|(-\Delta)^s \psi_{0,1}(x)|^2 \,\textrm{d}x 
			&\leq 
			(c_{N,s} |B_1|)^2 |\partial B_1|
			\int_2^{+\infty} \frac{R^{N-1}}{(R-1)^{2(N+2s)}} \, \textrm{d}R 
			\\
			&\leq
			C \int_2^{+\infty} \frac{1}{R^{N+4s+1}} \, \textrm{d}R 
			< 
			+\infty,
		\end{align*}
		where $C>0$ is a constant.
		That is, $(-\Delta)^s \psi_{0,1} \in L^2(\mathbb{R}^N \setminus B_2)$.
		
		When $|x| \in [1,2)$, we write $1-|y|^2 = (1+|y|)(1-|y|) \leq 2(|x|-|y|) \leq 2|x-y|$ 
		for $y \in B_1$ and obtain from \eqref{eq:delta-psi} that
		\begin{align*}
			|(-\Delta)^s \psi_{0,1}(x)|
			&\leq 
			2^p c_{N,s}
			\lim_{\varepsilon \to 0+} \int_{B_1 \setminus B_\varepsilon(x)} \frac{|x-y|^p}{|x-y|^{N+2s}} \, \textrm{d}y
			\leq
			2^p c_{N,s}
			\int_{B_3(x)} \frac{1}{|x-y|^{N+2s-p}} \, \textrm{d}y\\
			&= 
			2^p c_{N,s} |\partial B_1|
			\int_0^{3} \frac{R^{N-1}}{R^{N+2s-p}} \,\textrm{d}R
			=
			2^p c_{N,s} |\partial B_1|
			\int_0^{3} \frac{1}{R^{2s+1-p}} \,\textrm{d}R < +\infty,
		\end{align*}
		where the finiteness of the last integral follows from the assumption $p>2s$.
		In particular, we have $(-\Delta)^s \psi_{0,1} \in L^\infty(B_2 \setminus B_1)$, and, consequently, $(-\Delta)^s \psi_{0,1} \in L^2(\mathbb{R}^N \setminus B_1)$.
		On the other hand, the assertion \ref{lem:Psi:2} gives $(-\Delta)^s \psi_{0,1} \in L^\infty(B_1)$, which finally yields $(-\Delta)^s \psi_{0,1} \in L^2(\mathbb{R}^N)$.
		
		\ref{lem:Psi:5} 
		The assertion follows from  \cite[Theorem~5.3]{kwasnic} (the equivalence of $\psi_{0,1} \in \mathscr{D}(L_I,L^2(\mathbb{R}^N))$ and $\psi_{0,1} \in \mathscr{D}(L_Q,L^2(\mathbb{R}^N))$), since $\psi_{0,1} \in L^2(\mathbb{R}^N)$ and, by the assertion \ref{lem:Psi:1}, $(-\Delta)^s \psi_{0,1} \in L^2(\mathbb{R}^N)$.
		Here, we have $\psi_{0,1} \in L^2(\mathbb{R}^N)$ since $\psi_{0,1}$ is a continuous function with compact support, see \eqref{eq:subsol-1}.
		
		\ref{lem:Psi:3} 
		The assertion is given by \eqref{eq:delta-psi}.
		
		\ref{lem:Psi:4} 
		Since $p > 2s$ and $N/2 > 0$, the formula \cite[(15.1.20)]{abram} gives
		\begin{align*}
			{}_2F_1\left(s+\frac{N}{2}, -p+s; \frac{N}{2}; 1\right)
			= 
			\frac{\Gamma\left(\frac{N}{2}\right)\Gamma\left(p-2s\right)}
			{\Gamma\left(-s\right)\Gamma\left(p-s+\frac{N}{2}\right)}
			<0,
		\end{align*}
		where the last inequality follows from the fact that $\Gamma\left(-s\right)<0$ for $s \in (0,1)$.
		Therefore, there exists $\tau \in (0,1)$ such that 
		\begin{equation}\label{eq:deltapsi<0}
			(-\Delta)^s\psi_{0,1}(x)<0
			\quad \text{for any}~ |x| \in (\tau,1).
		\end{equation}
		On the other hand, for any $x \in \overline{B_\tau}$ we have the estimate
		\begin{align}
			\frac{(-\Delta)^s\psi_{0,1}(x)}{\psi_{0,1}(x)} 
			&=
			\kappa
			\frac{
				{}_2F_1\left(
				s+\frac{N}{2}, -p+s; \frac{N}{2}; |x|^2
				\right)}{(1-|x|^2)^p}\nonumber\\
			&\leq
			\kappa \max\limits_{x\in\overline{B_{\tau}}}
			\frac{
				{}_2F_1\left(
				s+\frac{N}{2}, -p+s; \frac{N}{2}; |x|^2
				\right)}{(1-|x|^2)^p}
			=: C
			< +\infty,
			\label{eq:deltapsi2}
		\end{align}
		thanks to the continuity of $x \mapsto \frac{(-\Delta)^s\psi_{0,1}(x)}{\psi_{0,1}(x)}$ in $\overline{B_\tau}$ following from the assertion \ref{lem:Psi:2}.
		Combining now \eqref{eq:deltapsi<0} with \eqref{eq:deltapsi2}, we derive Fitzsimmons' ratio estimate \eqref{eq:frac1} (cf.\ \cite[Lemma~9]{dyda1}).
		Observing that $\frac{(-\Delta)^s\psi_{0,1}(0)}{\psi_{0,1}(0)} = \kappa$, we conclude that $C \geq \kappa > 0$.
	\end{proof}

	Now we are ready to justify that $\underline{w}$ is a finite energy subsolution of \eqref{eq:P}.
	\begin{lemma}
		Let \ref{Q1} and \ref{Q2} be satisfied.
		Let $R>0$ and $x_0 \in \Omega$ be such that $B_R(x_0) \subset O$.
		Then, for any $p > \max\{1,2s\}$, $\underline{w}$ defined by \eqref{eq:subsol1} is a finite energy subsolution of \eqref{eq:P} for a sufficiently small $T>0$.
	\end{lemma}
	\begin{proof}
		Let us show that $\underline{w}$ satisfies the assumptions of Definition~\ref{def:subsuper}
		for a sufficiently small $T>0$.
		First, observe that \eqref{eq:subsol-1}, \eqref{eq:subsol0}, and \eqref{eq:subsol1} yield $\underline{w}(t,x) \to 0$ as $t \to 0+$ uniformly with respect to $x \in \mathbb{R}^N$, and hence
		$\|\underline{w}(t,\cdot)\|_{L^2(\Omega)} \to 0$ as $t \to 0+$.
		Second, we show that $\underline{w} \in L^2((0,T) \to X_0^s(\Omega))$ for any $T>0$.
		Since
		$$
		\int_0^T \|\underline{w}(t,\cdot)\|_{X_0^s(\Omega)}^2 \,\textrm{d}t 
		=
		\int_0^T \|\theta(t) \psi_{x_0,R}(\cdot)\|_{X_0^s(\Omega)}^2 \,\textrm{d}t 
		=
		\|\psi_{x_0,R}\|_{X_0^s(\Omega)}^2 \int_0^T \theta^2(t) \,\textrm{d}t
		$$
		and $\theta$ is bounded in $(0,T)$, 
		it is sufficient to verify that $\|\psi_{x_0,R}\|_{X_0^s(\Omega)}^2 < +\infty$.
		Recalling that $p>1$, we have $\psi_{x_0,R} \in C_0^1(\Omega)$. 
		Hence, arguing as in the proof of  \cite[Lemma~5.1]{Serv3}, we deduce that $\psi_{x_0,R} \in H^s(\mathbb{R}^N)$ and, consequently, $\|\psi_{x_0,R}\|_{X_0^s(\Omega)}$ is finite.
		Moreover, we have $\underline{w}_t \in L^2((0,T) \to (X_0^s(\Omega))^*)$ for any $T>0$, 
		and	$\underline{w}_t(t,x) = \theta'(t) \psi_{x_0,R}(x)$, see, e.g., \cite[Example~23.21]{zeid}.
		
		Our final goal is to prove the existence of $T>0$ such that
		\begin{equation}\label{eq:weak-sub1}
			\int_0^T \langle \underline{w}_t(t,\cdot),\varphi(t,\cdot) \rangle \, \textrm{d}t
			+
			\int_0^T \mathcal{E}(\underline{w}(t,\cdot), \varphi(t,\cdot)) \, \textrm{d}t
			\leq 
			\int_0^T \int_\Omega q(x) \underline{w}^{\alpha}(t,x) \varphi(t,x) \,\textrm{d}x \textrm{d}t
		\end{equation}
		for all nonnegative functions $\varphi \in L^2((0,T) \to X_0^s(\Omega))$. 
		This will be achieved by showing that  \eqref{eq:weak-sub1} is equivalent to
		\begin{equation}\label{eq:weak-sub11}
			\int_0^T \int_\Omega
			\left[ 
			\theta'(t) \psi_{x_0,R}(x) + \theta(t) (-\Delta)^s \psi_{x_0,R}(x) - q(x) \theta^\alpha(t) \psi_{x_0,R}^\alpha(x)
			\right]
			\varphi(t,x) \,\textrm{d}x\textrm{d}t \leq 0,
		\end{equation}
		where
		\begin{equation}\label{eq:weak-sub3}
			\theta'(t) \psi_{x_0,R}(x) + \theta(t) (-\Delta)^s \psi_{x_0,R}(x) - q(x) \theta^\alpha(t) \psi_{x_0,R}^\alpha(x) \leq 0
			\quad \text{in}~ [0,T) \times \Omega,
		\end{equation}
		with a sufficiently small $T>0$.
		First, we observe that \eqref{eq:weak-sub1} is 
		equivalent to 
		\begin{align}
			\int_0^T 
			\theta'(t)
			\langle  \psi_{x_0,R}(\cdot),\varphi(t,\cdot) \rangle \, \textrm{d}t
			&+
			\int_0^T 
			\theta(t) \mathcal{E}(\psi_{x_0,R}(\cdot), \varphi(t,\cdot)) \, \textrm{d}t
			\\
			\label{eq:weak-sub2}
			&\leq 
			\int_0^T \int_\Omega q(x) \theta^{\alpha}(t)
			\psi_{x_0,R}^\alpha(x)
			\varphi(t,x) \,\textrm{d}x \textrm{d}t.
		\end{align}
		Using now \eqref{eq:prod2}, we have
		\begin{equation}\label{eq:kwasnic0}
			\int_0^T 
			\theta'(t)
			\langle  \psi_{x_0,R}(\cdot),\varphi(t,\cdot) \rangle \, \textrm{d}t
			=
			\int_0^T 
			\theta'(t)
			\int_\Omega \psi_{x_0,R}(x) \varphi(t,x) \, \textrm{d}x\textrm{d}t.
		\end{equation}
		Moreover, by Lemma~\ref{lem:Psi}~\ref{lem:Psi:5},
		\begin{equation}\label{eq:kwasnic}
			\int_0^T 
			\theta(t) \mathcal{E}(\psi_{x_0,R}(\cdot), \varphi(t,\cdot)) \, \textrm{d}t
			=
			\int_0^T 
			\int_\Omega
			\theta(t)
			(-\Delta)^s\psi_{x_0,R}(x) \varphi(t,x) \, \textrm{d}x\textrm{d}t.
		\end{equation}
		Therefore, we deduce from \eqref{eq:kwasnic0} and \eqref{eq:kwasnic} that the inequalities \eqref{eq:weak-sub1} and \eqref{eq:weak-sub11} are equivalent. 
		
		It remains to prove that \eqref{eq:weak-sub3} is satisfied for a sufficiently small $T>0$.
		Note that
		\begin{equation}\label{eq:0psi1}
			0 < \psi_{x_0,R} \leq R \quad \text{in}~ B_R(x_0).
		\end{equation}
		Since $\theta \in C^1([0,+\infty))$ and $\theta(0)=0$, we can find $T>0$ such that $\theta^{1-\alpha}(t) \leq q_0 R^{\alpha-1}/(2C)$ for all $t \in [0,T)$, where $C>0$ is given by Lemma~\ref{lem:Psi}~\ref{lem:Psi:4}.
		In particular, we have
		\begin{equation}\label{eq:q012}
			\frac{q_0 R^{\alpha-1}}{2} + C \theta^{1-\alpha}(t) \leq q_0 R^{\alpha-1}
			\quad \text{for all}~ t \in (0,T).
		\end{equation}
		Using the equation for $\theta$ from \eqref{eq:subsol-2} and Lemma~\ref{lem:Psi}~\ref{lem:Psi:4}, 
		we estimate the left-hand side of \eqref{eq:weak-sub3} in $[0,T) \times B_R(x_0)$ as follows:
		\begin{align}
			\notag
			&\theta'(t) \psi_{x_0,R}(x) + \theta(t) (-\Delta)^s \psi_{x_0,R}(x) - q(x) \theta^\alpha(t)
			\psi_{x_0,R}^\alpha(x)
			\\
			\notag
			&\leq
			\frac{q_0R^{\alpha-1}}{2}\theta^\alpha(t) \psi_{x_0,R}(x) + C \theta(t)  \psi_{x_0,R}(x) - q(x) \theta^\alpha(t) \psi_{x_0,R}^\alpha(x)
			\\
			\label{eq:subsol:low1}
			&=
			\theta^\alpha(t) \psi_{x_0,R}(x)
			\left(
			\frac{q_0R^{\alpha-1}}{2}  + C \theta^{1-\alpha}(t)  - q(x) \psi_{x_0,R}^{\alpha-1}(x)
			\right).
		\end{align}
		In view of \eqref{eq:0psi1}, we have $\psi_{x_0,R}^{\alpha-1} \geq R^{\alpha-1}$ in $B_R(x_0)$, and hence, using \eqref{eq:q012} and \ref{Q1}, \ref{Q2}, we obtain from \eqref{eq:subsol:low1} that
		\begin{align*}
			\theta'(t) \psi_{x_0,R}(x) 
			&+ \theta(t) (-\Delta)^s \psi_{x_0,R}(x) - q(x) \theta^\alpha(t)
			\psi_{x_0,R}^\alpha(x)
			\\
			&\leq
			\theta^\alpha(t) \psi_{x_0,R}(x) R^{\alpha-1}
			\left(q_0  - q(x)\right) \leq 0
		\end{align*}
		for any $t \in [0,T)$ and $x \in B_R(x_0)$.
		On the other hand, if $t \in [0,T)$ and $x \in \Omega \setminus B_R(x_0)$, we get from Lemma~\ref{lem:Psi}~\ref{lem:Psi:3} that
		\begin{align*}
			&\theta'(t) \psi_{x_0,R}(x) + \theta(t) (-\Delta)^s \psi_{x_0,R}(x) - q(x) \theta^\alpha(t)
			\psi_{x_0,R}^\alpha(x) 
			=
			\theta(t)
			(-\Delta)^s \psi_{x_0,R}(x)
			\leq 0.
		\end{align*}
		Therefore, we proved the inequality \eqref{eq:weak-sub3}, and hence \eqref{eq:weak-sub1} is satisfied. 
	\end{proof}

	\subsection{Construction of a supersolution}\label{sec:supresol}
	
	Let us consider the fractional weighted Lane-Emden problem
	\begin{equation}\label{eq:D}
		\left\{
		\begin{aligned}
			(-\Delta)^s {u} &= q(x) |{u}|^{\alpha-1} {u}
			&&\text{in}~ \Omega,\\
			{u} &= 0 &&\text{in}~  \mathbb{R}^N \setminus  \Omega.
		\end{aligned}
		\right.
	\end{equation}
	\begin{definition}\label{def:sol-ellip}
		A function $u$ is called a finite energy solution of \eqref{eq:D} if $u \in X_0^s(\Omega)$
		and the equality
		\begin{equation}\label{eq:weak-D}
			\mathcal{E}(u, \varphi)
			=
			\int_\Omega q(x) |u(x)|^{\alpha-1} u(x) \varphi(x) \,\textrm{d}x
		\end{equation}
		is satisfied for all $\varphi \in X_0^s(\Omega)$.
	\end{definition}
	
	Using variational methods, it is not hard to show that \eqref{eq:D} possesses a nontrivial nonnegative finite energy solution. 
	Indeed, consider the energy functional $E$ associated with \eqref{eq:D}:
	$$
	E(u) 
	= 
	\frac{1}{2} \|u\|_{\X}^2
	-
	\frac{1}{\alpha+1} \int_\Omega q(x) |u|^{\alpha+1} \,\textrm{d}x,
	\quad u \in X_0^s(\Omega).
	$$
	Since $q \in L^\infty(\Omega)$ and $\alpha \in (0,1)$, the H\"older inequality and Remark~\ref{rem:reg} yield
	$$
	\int_\Omega q(x) |u|^{\alpha+1} \,\textrm{d}x
	\leq 
	C\|u\|_{\X}^{\alpha+1},
	\quad u \in X_0^s(\Omega).
	$$
	Thus, we see that $E$ is coercive, i.e., $E(u) \to +\infty$ as $\|u\|_{X_0^s(\Omega)} \to +\infty$. 
	Note also that $E$ is weakly lower semicontinuous, as it also follows from Remark~\ref{rem:reg}. 
	Therefore, by the direct method of the calculus of variations, we get the existence a global minimizer $\hat{u}$ of $E$, which is a finite energy solution of \eqref{eq:D}.
	Taking any $u \in X_0^s(\Omega)$ such that $\int_\Omega q(x) |u|^{\alpha+1} \,\textrm{d}x > 0$, we observe that $E(tu)<0$ for any sufficiently small $t>0$. 
	Consequently, the global minimizer $\hat{u}$ is nonzero in $\Omega$.
	In view of the evenness of $E$, $\hat{u}$ can be assumed nonnegative, and hence the strong maximum principle (see, e.g., \cite[Proposition~A.2]{FL}) guarantees that $\hat{u}>0$ a.e.\ in $\Omega$.
	Note that \cite[Proposition~A.2]{FL} is applicable since the space $\mathcal{D}_0^{s,2}(\Omega)$ coincides with $\X$ under our assumptions on $\Omega$, see \cite[Remark~2.1]{FL}.
	In general, we refer to \cite{FL} for a thorough discussion on the fractional Lane-Emden problem \eqref{eq:D} in the unweighted case $q = 1$ a.e\ in $\Omega$. 
	Moreover, we have $\hat{u} \in L^\infty(\Omega)$, see, e.g., \cite[Lemma~1.1]{CGH2} or \cite[Proposition~3.1]{FL}.
	
	The positive global minimizer $\hat{u}$ of $E$ generates a stationary bounded finite energy solution (and hence supersolution) of \eqref{eq:P}. 
	Indeed, considering $\bar{u}(t,\cdot)=\hat{u}(\cdot)$ for all $t \in [0,T]$, we see that $\bar{u}_t = 0$ in $L^2((0,T) \to (X_0^s(\Omega))^*)$ (cf.\ \cite[Example~23.21]{zeid}), and all other assumptions of Definition~\ref{def:super} can be derived straightforwardly.

	\subsection{Existence of a solution}\label{sec:existence}
	
	In this section, we prove the existence of a finite energy solution of \eqref{eq:P} in between the subsolution $\underline{u}$ and the supersolution $\bar{u}$ constructed in the previous subsections. 
	The argument is based on the rather classical monotone iteration method. For the sake of clarity, we provide details. 
	
	Set $u^0 := \underline{u}$.
	For any $n \in \mathbb{N}$, let $u^{n}$ be a unique finite energy solution of the following problem:
	\begin{equation}\label{eq:un}
		\left\{
		\begin{aligned}
			u^{n}_t
			+
			(-\Delta)^s u^{n} 
			&= 
			q(x) |u^{n-1}(t,x)|^{\alpha-1}u^{n-1}(t,x)
			&&\text{in}~ (0,T) \times \Omega,\\
			u^{n} &= 0 &&\text{in}~  (0,T) \times (\mathbb{R}^N \setminus  \Omega),\\
			u^{n} &= 0 &&\text{at}~ t=0.
		\end{aligned}
		\right.
	\end{equation}
	We recall from Section~\ref{sec:preliminaries} that the existence and uniqueness of $u^{n}$ is
	obtained, for example, by \cite[Theorem~26]{LPPS}, by noting that for a.e.\ $t \in (0,T)$ the right-hand side of the equation in \eqref{eq:un} can be naturally identified with some $g(t) \in L^2((0,T) \to (X_0^s(\Omega))^*)$ by the equality
	\begin{equation}\label{eq:gxi}
		\langle g(t), \xi \rangle
		=
		\int_\Omega 
		q(x) |u^{n-1}(t,x)|^{\alpha-1}u^{n-1}(t,x) \, \xi(x)
		\,\textrm{d}x
		\quad \text{for all}~ \xi \in \X.
	\end{equation}
	Moreover, as in \eqref{eq:weak-f-eq2}, we have
	$J{u^n}(t) \in (X_0^s(\Omega))^*$, $\|J{u^n}(t)\|_{(X_0^s(\Omega))^*} = \|u^n(t,\cdot)\|_{\X}$, and 
	\begin{equation}\label{eq:weak-f-eq3}
		\langle 
		u^n_t(t)
		+
		J{u^n}(t)
		-
		g(t),
		\xi
		\rangle 
		=
		0
		\quad
		\text{for a.e.}~ t \in (0,T)
		~\text{and all}~
		\xi \in \X.
	\end{equation}

	Using inductive arguments, it is not hard to show that 
	$u^n \geq u^{n-1} \geq 0$ for all $n \geq 1$.
	Indeed, let us denote $w^n = u^n - u^{n-1}$.
	Recalling that $u^0 = \underline{u}$ is a finite energy subsolution of \eqref{eq:P}, we see that $w^1$ satisfies
	\begin{equation*}
		\left\{
		\begin{aligned}
			w^{1}_t
			+
			(-\Delta)^s w^{1} & \geq 0
			&&\text{in}~ (0,T) \times \Omega,\\
			w^{1} &= 0 &&\text{in}~  (0,T) \times (\mathbb{R}^N \setminus  \Omega),\\
			w^{1} &= 0 &&\text{at}~ t=0.
		\end{aligned}
		\right.
	\end{equation*}
	Thus, \cite[Theorem~26]{LPPS} ensures that $w^1 \geq 0$.
	Suppose now that $w^{n} \geq 0$  for some $n \geq 2$ and let us justify that  $w^{n+1} \geq 0$.
	Since, by construction,  $w^{n+1}$ satisfies
	\begin{equation}\label{eq:wn1}
		\left\{
		\begin{aligned}
			w^{n+1}_t
			+
			(-\Delta)^s w^{n+1} &= q(x) 
			\left(
			|u^{n}|^{\alpha-1}u^{n}
			-
			|u^{n-1}|^{\alpha-1}u^{n-1}
			\right)
			&&\text{in}~ (0,T) \times \Omega,\\
			w^{n+1} &= 0 &&\text{in}~  (0,T) \times (\mathbb{R}^N \setminus  \Omega),\\
			w^{n+1} &= 0 &&\text{at}~ t=0,
		\end{aligned}
		\right.
	\end{equation}
	and since $s \to |s|^{\alpha-1}s$ is an increasing function in $\mathbb{R}$ and $q \geq 0$ in $\Omega$ by \ref{Q1}, the right-hand side of the equation in \eqref{eq:wn1} is nonnegative, which implies the desired nonnegativity of $w^{n+1}$. 
	Therefore, we conclude that  $u^{n} \geq u^{n-1} \geq 0$ for every $n \geq 1$.
	
	Using similar arguments based on the maximum principle, 
	it can be easily shown that $u^n \leq \overline{u}$ for any $n \geq 1$.
	In particular,  $\{u^n\}$ is bounded in $L^2((0,T)\times \Omega)$.
	Considering the pointwise limit $u(t,x) := \lim_{n\to+\infty} u^n(t,x)$ for a.e.\ $(t,x) \in (0,T) \times \Omega$, we apply the Beppo Levi theorem and deduce that $u \in L^2((0,T)\times \Omega)$ and $u^n \to u$ in $L^2((0,T)\times \Omega)$. 
	Note that such $u$ is uniquely determined by $\{u^n\}$ and	we have $\underline{
		u} \leq u \leq \overline{u}$.
	Moreover, the dominated convergence theorem gives
	\begin{equation}\label{eq:conv-eq-3}
		\int_0^T
		\int_\Omega 
		q(x) (u^{n-1}(t,x))^{\alpha} \varphi(t,x) \, \textrm{d}x\textrm{d}t
		\to 
		\int_0^T
		\int_\Omega 
		q(x) (u(t,x)))^{\alpha} \varphi(t,x) \, \textrm{d}x\textrm{d}t
	\end{equation}
	for any $\varphi \in L^2((0,T) \to X_0^s(\Omega))$.
	
	Let us prove that $u$ is a finite energy solution of \eqref{eq:P}.
	For this purpose, we first show that $\{u^n\}$ is bounded in $L^2((0,T) \to X_0^s(\Omega))$.
	Since $u_n$ is a finite energy solution of \eqref{eq:un}, we use $u^n$ itself as a test function and obtain
	\begin{align*}
		\int_0^T \langle u^n_t(t),u^n(t,\cdot) \rangle \, \textrm{d}t
		+
		\int_0^T \mathcal{E}(u^n(t,\cdot), u^n(t,\cdot)) \, \textrm{d}t
		=
		\int_0^T
		\int_\Omega 
		q(x) (u^{n-1}(t,x))^{\alpha}u^{n}(t,x) \, \textrm{d}x\textrm{d}t.
	\end{align*}
	Therefore, we get
	\begin{align}
		\hspace{-2em}
		\|u^n\|_{L^2((0,T) \to X_0^s(\Omega))}^2
		&\equiv
		\int_0^T \mathcal{E}(u^n(t,\cdot), u^n(t,\cdot)) \, \textrm{d}t
		\\
		\label{eq:un-bound}
		&=
		\int_0^T
		\int_\Omega 
		q(x) (u^{n-1}(t,x))^{\alpha}u^{n}(t,x) \, \textrm{d}x\textrm{d}t
		-
		\int_0^T \langle u^n_t(t),u^n(t,\cdot) \rangle \, \textrm{d}t.
	\end{align}
	Using the integration by parts (see, e.g., \cite[Lemma~7.3]{roub}) and recalling that $u^n(0,\cdot)=0$ in $\Omega$ and $u^n \geq 0$ for all $n$, we have
	\begin{align*}
		\int_0^T \langle u^n_t(t),u^n(t,\cdot) \rangle \, \textrm{d}t
		=
		\frac{1}{2}
		\int_\Omega (u^n(T,x))^2 \, \textrm{d}x
		-
		\frac{1}{2}
		\int_\Omega (u^n(0,x))^2 \, \textrm{d}x
		\geq
		0.
	\end{align*}
	Thus, in view of the upper bound $u^n \leq \bar{u}$ for all $n$, we deduce from \eqref{eq:un-bound} that
	\begin{equation}\label{eq:uninf}
		\|u^n\|_{L^2((0,T) \to X_0^s(\Omega))}^2
		\leq 
		\int_0^T
		\int_\Omega 
		q(x) (\bar{u}(t,x))^{\alpha+1} \, \textrm{d}x\textrm{d}t
		=
		\text{const} <+\infty.
	\end{equation}
	
	The boundedness of
	$\{u^n\}$ in $L^2((0,T) \to X_0^s(\Omega))$  implies the existence of $w \in L^2((0,T) \to X_0^s(\Omega))$ such that
	$u^n \to w$ weakly in $L^2((0,T) \to X_0^s(\Omega))$, up to a subsequence.
	Since the embedding $L^2((0,T) \to X_0^s(\Omega)) \hookrightarrow L^2((0,T) \times \Omega)$ is continuous, 
	$u^n \to w$ weakly in $L^2((0,T) \times \Omega)$.
	Recalling that $u^n \to u$ strongly in $L^2((0,T) \times \Omega)$,
	we deduce that $w=u$.
	Since the pointwise limit $u$ is unique by construction, we conclude that $u^n \to u$ weakly in $L^2((0,T) \to X_0^s(\Omega))$
	over the whole sequence. 
	Therefore, recalling from Section~\ref{sec:preliminaries} that $\int_0^T \mathcal{E}(v(t,\cdot), w(t,\cdot))\,\textrm{d}t$ is a scalar product in the Hilbert space $L^2((0,T) \to X_0^s(\Omega))$, we get
	\begin{equation}\label{eq:conv-eq-2}
		\int_0^T \mathcal{E}(u^n(t,\cdot), \varphi(t,\cdot)) \, \textrm{d}t
		\to 
		\int_0^T \mathcal{E}(u(t,\cdot), \varphi(t,\cdot)) \, \textrm{d}t
	\end{equation}
	for any $\varphi \in L^2((0,T) \to X_0^s(\Omega))$.
	
	Let us show now that $\{u^n_t\}$ is bounded in $L^2((0,T) \to (X_0^s(\Omega))^*)$.
	We have, by definition,
	\begin{align*}
		\|u^n_t\|_{L^2((0,T) \to (X_0^s(\Omega))^*)}^2
		=
		\int_0^T
		\|u^n_t(t)\|_{(X_0^s(\Omega))^*}^2 \,\textrm{d}t
		=
		\int_0^T
		\left(
		\sup_{\xi \in X_0^s(\Omega) \setminus \{0\}}
		\left\{
		\frac{\langle u^n_t(t),\xi \rangle}{\|\xi\|_{X_0^s(\Omega)}}
		\right\}
		\right)^2\textrm{d}t.
	\end{align*}
	Thanks to \eqref{eq:gxi}, \eqref{eq:weak-f-eq3}, and \ref{Q1}, 
	for a.e.\ $t \in (0,T)$ and any $\xi \in X_0^s(\Omega)$ we obtain
	\begin{align*}
		\langle u^n_t(t),\xi \rangle
		&=
		-\langle J{u^n}(t), \xi \rangle
		+
		\int_\Omega q(x) (u^{n-1}(t,x))^\alpha \xi(x) \,\textrm{d}x
		\\
		&\leq
		\|u^n\|_{X_0^s(\Omega)}
		\|\xi\|_{X_0^s(\Omega)}
		+
		\|q\|_{L^{\infty}(\Omega)}
		\|(u^{n-1})^{\alpha}\|_{L^{2}(\Omega)}
		\|\xi\|_{L^{2}(\Omega)}
		\\
		&\leq
		\left(
		\|u^n\|_{X_0^s(\Omega)}
		+
		C \|q\|_{L^{\infty}(\Omega)}
		\|\overline{u}^{\alpha}\|_{L^{2}(\Omega)}
		\right)
		\|\xi\|_{X_0^s(\Omega)},
	\end{align*}
	where $C>0$ is given by the continuous embedding $X_0^s(\Omega) \hookrightarrow L^2(\Omega)$.
	Then, using the simple estimate $(a+b)^2 \leq 2a^2+2b^2$ and the first inequality in \eqref{eq:uninf}, we arrive at
	$$
	\|u^n_t\|_{L^2((0,T) \to (X_0^s(\Omega))^*)}^2
	\leq 
	2 \int_0^T
	\int_\Omega 
	q(x) (\bar{u}(t,x))^{\alpha+1} \, \textrm{d}x\textrm{d}t
	+
	2 T C^2 \|q\|_{L^{\infty}(\Omega)}^2
	\|\overline{u}^{\alpha}\|_{L^{2}(\Omega)}^2.
	$$
	In view of the boundedness of $\overline{u}$, we conclude that $\{u^n_t\}$ is bounded in $L^2((0,T) \to (X_0^s(\Omega))^*)$.
	This implies that $\{u^n_t\}$ converges weakly in 
	$L^2((0,T) \to (X_0^s(\Omega))^*)$ to some $v \in L^2((0,T) \to (X_0^s(\Omega))^*)$.
	Then, \cite[Proposition~23.19]{zeid} yields $v = u_t$, and hence we have
	\begin{equation}\label{eq:conv-eq-1}
		\int_0^T \langle u^{n}_t(t),\varphi(t,\cdot) \rangle \, \textrm{d}t
		\to 
		\int_0^T \langle u_t(t),\varphi(t,\cdot) \rangle \, \textrm{d}t
	\end{equation}
	for any $\varphi \in L^2((0,T) \to X_0^s(\Omega))$.
	
	Finally, combining \eqref{eq:conv-eq-3}, \eqref{eq:conv-eq-2}, and \eqref{eq:conv-eq-1}, we conclude that $u$ is a nontrivial nonnegative bounded finite energy solution of the problem \eqref{eq:P}.
	
	\smallskip
	In order to finish the proof, it remains to justify that positive time shifts of $u$ also generate solutions of \eqref{eq:P}.
	Taking into account that the equation in \eqref{eq:P} is autonomous in time, it is enough to show that the extension of $u$ by zero to $(-\tau,T)$ (denoted again by $u$) for a given time shift $\tau>0$ 
	is a finite energy solution of \eqref{eq:P} on $(-\tau,T)$.
	It is clear from Definition \ref{def:sol-parab} that the only fact deserving a justification is that the function $w$, defined as an extension of $u_t$ by zero to $(-\tau,T)$, is a distributional derivative of $u$.
	For this purpose, it is sufficient to show that, for an arbitrary fixed $v \in \X$, the function $t \mapsto \langle w(t), v\rangle$ is a weak derivative of  
	$t \mapsto \int_\Omega u(t,x) v(x) \,\textrm{d}x$ a.e.\ in $(-\tau,T)$, 
	that is,
	\begin{equation}\label{eq:weakder1}
		\int_{0}^T \left(\int_\Omega u(t,x) v(x) \,\textrm{d}x\right) \varphi'(t) \,\textrm{d}t
		=
		-\int_{0}^T \langle w(t), v\rangle \varphi(t) \,\textrm{d}t
		\quad \text{for any}~ \varphi \in C_0^\infty(-\tau,T),
	\end{equation}
	see \cite[Proposition~23.20~(b) and Remark~23.17]{zeid}.
	It follows from Definition \ref{def:sol-parab} that for a.e.\ $t \in (0,T)$, $t \mapsto \langle w(t), v\rangle$ is a weak derivative of  
	$t \mapsto \int_\Omega u(t,x) v(x) \,\textrm{d}x$ and it belongs to $L^2(0,T)$. Hence, the function $t \mapsto \int_\Omega u(t,x) v(x) \,\textrm{d}x$ is absolutely continuous on $[0,T]$.
	Therefore, recalling that $\|u(t,\cdot)\|_{L^2(\Omega)} \to 0$ as $t \to 0+$, we conclude from the standard integration by parts for absolutely continuous functions that \eqref{eq:weakder1} is satisfied, which completes the proof.

	\section{Proof of Proposition~\ref{prop:uniq}}\label{sec:uniq}
	Let $\alpha \geq 1$ and let $u$ be a finite energy solution of \eqref{eq:P} for some $T>0$.
	In the case $\alpha > 1$, we additionally require the boundedness of $u$. 
	Testing \eqref{eq:P} 
	with $u(t,x)\chi_{(0,\sigma)}(t)$ for any $\sigma \in (0,T)$, where $\chi_{(0,\sigma)}$ is the characteristic function of the interval $(0,\sigma)$, we obtain
	\begin{equation}\label{eq:al>11}
		\int_0^\sigma \langle u_t(t),u(t,\cdot) \rangle \, \textrm{d}t
		+
		\int_0^\sigma \mathcal{E}(u(t,\cdot), u(t,\cdot)) \, \textrm{d}t
		=
		\int_0^\sigma \int_\Omega q(x) |u(t,x)|^{\alpha+1} \,\textrm{d}x \textrm{d}t.
	\end{equation}
	Performing the integration by parts (see, e.g., \cite[Lemma~7.3]{roub}), we get
	\begin{equation}\label{eq:al>112}
		\int_0^\sigma \langle u_t(t),u(t,\cdot) \rangle \, \textrm{d}t
		=
		\frac{1}{2} \int_\Omega u^2(\sigma,x) \,\textrm{d}x
		-
		\frac{1}{2} \int_\Omega u^2(0,x) \,\textrm{d}x
		=
		\frac{1}{2}\|u(\sigma,\cdot)\|_{L^2(\Omega)}^2,
	\end{equation}
	and, by \eqref{eq:Euu}, 
	\begin{equation}\label{eq:al>113}
		\mathcal{E}(u(t,\cdot), u(t,\cdot)) = \|u(t,\cdot)\|_{\X}^2 
		\geq 0 
		\quad \text{for a.e.}~ t \in (0,\sigma).
	\end{equation}
	By the assumption \ref{Q1} and the boundedness of $u$ in the case $\alpha > 1$, we obtain the existence of a constant $C_0>0$, independent of $\sigma$, such that 
	\begin{equation}\label{eq:al>114}
		\int_\Omega q(x) |u(t,x)|^{\alpha+1} \,\textrm{d}x
		\leq
		\|q\|_{L^{\infty}(\Omega)}
		\|u(t,\cdot)\|_{L^{\infty}(\Omega)}^{\alpha-1} \|u(t,\cdot)\|_{L^{2}(\Omega)}^2 
		\leq 
		C_0 \|u(t,\cdot)\|_{L^{2}(\Omega)}^2
	\end{equation}
	for a.e.\ $t \in (0,\sigma)$.
	Here we note that if $\alpha=1$, then the term $\|u(t,\cdot)\|_{L^{\infty}(\Omega)}^{\alpha-1}$ does not appear in \eqref{eq:al>114}.
	Substituting \eqref{eq:al>112}, \eqref{eq:al>113}, and \eqref{eq:al>114} into \eqref{eq:al>11}, we deduce that
	$$
	\frac{1}{2}\|u(\sigma,\cdot)\|_{L^2(\Omega)}^2
	\leq 
	C_0 \int_0^\sigma  \|u(t,\cdot)\|_{L^2(\Omega)}^2\,\textrm{d}t.
	$$
	Thus, recalling that $u \in C([0,T] \to L^2(\Omega))$ (see Section~\ref{sec:preliminaries}), the function $y(\sigma):=\|u(\sigma,\cdot)\|_{L^2(\Omega)}^2$, $\sigma \in (0,T)$, is continuous and
	satisfies
	$$
	y(\sigma) \leq 2\, C_0 \int_0^\sigma  y(t)\,\textrm{d}t.
	$$
	The Gr\"onwall inequality \cite[Eq.~(1.66), with $C:=0$, $a:= 2\,C_0$, and $b:=0$]{roub}
	yields
	$y(\sigma)\leq 0$. Consequently, we have  
	$\|u(\sigma,\cdot)\|_{L^2(\Omega)}^2=0$ for any $\sigma \in (0,T)$, and hence the uniqueness follows.
	
	\begin{remark}
		It is not hard to see from the arguments above that the result of Proposition~\ref{prop:uniq}~\ref{prop:uniq:2} remains valid for the problem \eqref{eq:P} with a more general nonlinearity of the form $q(x)f(u)$, where $f$ is Lipschitz continuous, since in this case an estimate corresponding to \eqref{eq:al>114} holds true.
	\end{remark}

\end{document}